\documentclass[oneside,11pt]{amsart}

\newskip\stdskip                      
\usepackage{pinlabel}  
\usepackage{pictexwd,dcpic}

\usepackage{amscd,latexsym,amssymb}
\usepackage{bm}

  \usepackage{hyperref}  
  \hypersetup{%
  bookmarksnumbered=true,%
  bookmarks=true,%
  colorlinks=true,%
  linkcolor=blue,%
  citecolor=blue,%
  filecolor=blue,%
  menucolor=blue,%
  pagecolor=blue,%
  urlcolor=blue,%
  pdfnewwindow=true,%
  pdfstartview=FitBH}

\newtheorem{thm}{Theorem}[section]
\newtheorem{mthm}{Theorem}
\newtheorem{cor}[thm]{Corollary}
\newtheorem{lem}[thm]{Lemma}
\newtheorem{prop}[thm]{Proposition}

\theoremstyle{definition}
\newtheorem{defin}[thm]{Definition}

\theoremstyle{definition}

\newtheorem{exm}[thm]{Example}

\newtheorem{remark}[thm]{Remark}

\theoremstyle{remark}

\def\co{\colon\thinspace}

\begin{document}
\title{Norm minima in certain Siegel leaves}
\author{Li Cai}
\address{School of Mathematics and Systems Science, 
Chinese Academy of Sciences, Beijing 100190, China}
\email{l-cai@amss.ac.cn}

\subjclass[2010]{Primary 57R30; Secondary 57R70, 05E45}

\keywords{foliation, moment-angle manifold, 
simplicial complex}

\begin{abstract}    
	In this paper we shall illustrate that each polytopal moment-angle complex
	can be understood as the intersection of the minima of corresponding Siegel 
	leaves and the unit sphere, with respect to the maximum norm. Consequently, 
	an alternative proof of a rigidity theorem of Bosio and Meersseman is obtained; 
	as piecewise linear manifolds, polytopal real moment-angle complexes can 
	be smoothed in a natural way. 
\end{abstract}

\maketitle

\section{Introduction}
An \emph{admissible configuration} of $m$ complex vectors in 
$\mathbb{C}^{d/2}$ ($m>d$ with $d$ even)
satisfying so called \emph{Siegel} and \emph{weak hyperbolicity} 
conditions (cf.~\cite[p.~82]{Mee00}; see Section \ref{sec:nota}
for a real analogue), 
gives rise to a free action on $\mathbb{C}^m$ via 
exponential functions. There are two types of 
leaves in the holomorphic foliation 
given by this action: a leaf is of 
\emph{Siegel type} if the origin is not 
in its closure, otherwise it is said of Poincar\'{e} type. 

These objects originated in the work \cite{CKP78} 
of C.~Camacho, N.~Kuiper and J.~Palls 
on the complex analogue of a 
dynamical system for which the real version appeared in 
an earlier work of Poincar\'{e},
and later have been developed and generalized by S.~L\'{o}pez de Medrano, 
A.~Verjovsky and L.~Meersseman (cf.~\cite{LV97}, \cite{Mee00}).  
From their works, 
the projectivization of the minima of 
all Siegel leaves, with respect to the Euclidean norm, 
can be endowed with the structure of a 
compact, complex $(m-d/2-1)$-manifold $C^{\infty}$-embedded in $\mathbb{C}P^{m-1}$, 
which is not symplectic except the trivial case.
This class of complex manifolds is now named as \emph{$LVM$ manifolds}.

On the other hand, with a direct calculation, 
the space of minima of all Siegel leaves can 
be described by $d$ \emph{real quadrics} 
arising from the given configuration in $\mathbb{R}^d$, 
whose intersection with unit Euclidean
sphere in $\mathbb{C}^m$ is transverse hence
is a smooth manifold  of real dimension $2m-d-1$. F.~Bosio and L.~Meersseman 
observed that this method also works for odd $d$, and call these 
manifolds embedded in spheres as \emph{links} in \cite{BM06}.

This special class of links is a model
for \emph{polytopal moment-angle manifolds}.
In general their topology is known to be 
complicated (cf.~\cite{BM06}, \cite{GL13}), for instance,
arbitrary \emph{torsion} can appear in the cohomology, 
as well as non-vanishing triple \emph{Massey products} (cf.~\cite{Bas03}, \cite{DS07});
 in the case $d=2$, 
the classification work \cite{LdM89} by S.~L\'{o}pez de Medrano 
shows that they are diffeomorphic to a triple product of spheres 
or to the connected sum of sphere products. 
An important way to understand them is that  
they inherit the natural $(S^1)^{m}$-action on $\mathbb{C}^m$, 
with each quotient space homeomorphic (as manifolds with corners) 
to a simple convex polytope.
Via the \emph{basic construction} originating 
from \emph{reflection group theory} and then generalized 
by M.~W.~Davis and T.~Januszkiewicz 
in their influential work \cite{DJ91}, 
each link discussed above is homeomorphic to a \emph{moment-angle complex}
(named in \cite{BP02}), 
i.e.~a \emph{polyhedral product} 
with pairs $(D^2,S^1)$ corresponding to the boundary complex of a simplicial
polytope. 

 The polyhedral product model has been studied in detail and generalized by 
 V.~Buchstaber and T.~Panov in \cite{BP02}. Later a more categorical treatment by 
 A.~Bahri, M.~Bendersky, F.~R.~Cohen and S.~Gitler in their work 
 \cite{BBCG10} provides a penetrating viewpoint from homotopy theory. 
 
 These spaces have spawned a large body of the work notably
 with Davis-Januszkiewicz \cite{DJ91} on quasi-toric varieties, 
 Buchstaber-Panov \cite{BP02} on moment-angle complexes, Goresky-MacPherson 
 \cite{GM88} on complements
 of complex arrangements, S.~L\'{o}pez de Medrano \cite{LdM89} 
 on the topology of these varieties, as well as many others. The interconnections
 between these subjects is developed in the beautiful book 
 \cite{BP14} by Buchstaber-Panov.

 The objective of this paper is to show that, for 
 an admissible configuration of $m$ 
 real vectors in $\mathbb{R}^d$ \emph{whose centroid is 
 located at the origin}, the corresponding foliation
 provides a direct
 relation between the model of links and 
 the model of polyhedral products: there are continuous 
 paths in the space of the union all Siegel leaves (which is the 
 \emph{complement of a coordinate subspace arrangement} in 
 $\mathbb{C}^m$), such that each point of the link is connected 
 by a path to a unique point in the respective moment-angle complex, 
 yielding a homeomorphism between them.
 Every path is parameterized by real numbers 
 $p\in[1,\infty)$, with each $p$ 
  associated to the intersection of the $L^p$-norm minima 
  in the Siegel leaves and the $L^p$-norm unit sphere in $\mathbb{C}^m$, 
  which is a topological manifold homeomorphic to the link. In 
  this way, we can understand each polytopal moment-angle complex
  as the intersection of the unit sphere and the minima of all 
  Siegel leaves, with respect to the $L^{\infty}$-norm.
 
 This paper develops a more analytic approach to these spaces
 in the spirit of the work \cite{BM06} by Bosio and Meersseman.

  I would like to thank my Ph.D.~supervisor, Professor 
  Osamu Saeki for discussions.

\section{Notations and main results}\label{sec:nota}
Let $A=(A_1,A_2,\dots,A_m)$ be an $m$-tuple of  
vectors in $\mathbb{R}^d$, with $m>d\geq 0$ ($A_i\equiv 0$ when $d=0$); 
occasionally we treat such a tuple as a $(d\times m)$-matrix.
Denote by $[m]$ the set $\{1,2,\dots,m\}$, and for $I\subset[m]$, let 
$A(I)$ be the subtuple $(A_i)_{i\in I}$ and $\mathrm{conv}A$ 
(resp.~$\mathrm{conv}A(I)$) the convex hull of vectors from $A$
(resp.~from $A(I)$). 

We say that $A$ is \emph{admissible}, if it satisfies the following 
two conditions (cf.~\cite[Lemma 0.3]{BM06}):
 \begin{enumerate}
\item [$*_1$] (Siegel condition) $\bm{0}\in\mathrm{conv}A$;
\item [$*_2$] (weak hyperbolicity condition) if $\bm{0}\in\mathrm{conv}A(I)$, 
 then we have $\mathrm{card}(I)> d$ 
 (where $\mathrm{card}$ refers to the cardinality).
\end{enumerate}
    
Up to Section \ref{sec.3.},
we always assume that $A$ is admissible.

 Let $\mathbb{R}_{>0}$ be the set of positive real numbers, in which 
 $p\geq 1$ is a real number. For each $z=(z_i)_{i=1}^m\in\mathbb{C}^m$, 
 denote by $\lVert z\rVert_{p}$ its $L^p$-norm, namely  
 $\lVert z\rVert_{p}=(\sum_{i=1}^m|z_i|^p)^{\frac{1}{p}}$ where 
 $|z_i|=\sqrt{z_i\bar{z}_i}$.

 With respect to an $m$-tuple $A$, there is a smooth 
 \emph{foliation} $\mathcal{F}$ of $\mathbb{C}^m$ 
 given by the orbits of the action 
 \begin{equation}
   \begindc{\commdiag}[2]
   \obj(0,7)[aa]{$F\co\mathbb{C}^m
   \times\mathbb{R}^d$}
   \obj(40,7)[bb]{$\mathbb{C}^m$}
   \obj(0,0)[cc]{$(z,T)$}
   \obj(40,0)[dd]{$(z_ie^{\langle A_i,T \rangle})_{i=1}^m$.}
   \mor{aa}{bb}{}
   \mor{cc}{dd}{}[+1,6]
 \enddc\label{Def.F.}
 \end{equation}
 For each $z\in\mathbb{C}^m$, let 
 $L_z$ be the \emph{leaf} passing through $z$. $L_z$ is 
 called a \emph{Siegel leaf} if $\bm{0}$ is not 
 in its closure, otherwise we say the leaf $L_z$ is of Poincar\'{e}
 type. It follows that the union of all Siegel leaves 
 can be described by the set 
 (cf.~\cite{CKP78}, \cite{MV04}, \cite{BM06}) 
 \begin{equation}
   \mathcal{S}_A=\{z\in\mathbb{C}^m\mid \bm{0}\in\mathrm{conv}A(I_z)\}, 
  \label{Def.SA.}
\end{equation}
 where $I_z$ is the set of non-zero entries for $z=(z_i)_{i=1}^m$, 
 i.e.~$I_z=\{i\in[m]\mid |z_i|\not=0\}$. With an argument involving 
 foliations, complex analysis and the convexity, 
 the following fact is a combination of the works mentioned above, 
 which is our starting point: 
 \begin{mthm}[cf.~{\cite[Lemma 0.8, pp.~61--62]{BM06}}]
   \label{funda.thm.}
   For each $z\in\mathcal{S}_A$, there
   is a unique point $f_2(z)$ in the leaf $L_z$, such that 
   its $L^2$-norm $\lVert f_2(z)\rVert_2$ is minimal and positive.  
   The foliation $\mathcal{F}$ is trivial when restricted to $\mathcal{S}_A$,
   and 
  \begin{equation*}
   \begindc{\commdiag}[15]
   \obj(0,1)[aa]{$\Phi_A(2)\co X_A(2)
	 \times\mathbb{R}^d\times\mathbb{R}_{>0}$}
   \obj(7,1)[bb]{$\mathcal{S}_A$}
   \obj(0,0)[cc]{$(z,T,r)$}
   \obj(7,0)[dd]{$r(z_ie^{\langle A_i,T \rangle})_{i=1}^m$}
   \mor{aa}{bb}{}
   \mor{cc}{dd}{}[+1,6]
   \enddc
 \end{equation*}
   is a global diffeomorphism, where $X_A(2)$ is given 
   by the transverse intersection
     \begin{equation}
	  \begin{cases}
		\sum_{i=1}^mA_i|z_i|^{2}=\bm{0}, \\
		\lVert z\rVert_{2}=1,
	  \end{cases}\label{def.x2.}	  
	  \end{equation}
	   thus is a smooth manifold.
\end{mthm}
It follows  that there is a smooth function 
   \begin{equation}
	T_2\co\mathcal{S}_A\to\mathbb{R}^d, 
	\text{ $s.t.$ } f_2(z)=F(z,T_2(z)), 
	 \label{Def.T2.}
   \end{equation}
and after differentiating ${F}(z,T)$ with 
respect to $T\in\mathbb{R}^d$, one easily checks that 
the critical point corresponding to the minimum satisfies
 \begin{equation}
 \sum_{i=1}^mA_i|z_i|^2e^{2\langle A_i,T \rangle}=\bm{0},
 \label{diff.mini.}
 \end{equation}
in which $T_2(z)$ is the unique solution. Moreover, 
$f_2/\lVert f_2\rVert_2\co\mathcal{S}_A\to X_A(2)$ is a 
 smooth retraction. 

 Following their approach, we consider 
 the space of $L^p$-norm minima of those Siegel leaves. 
 Our first main theorem is the 
 following, whose proof is based on some real analysis and 
 will be given in Section \ref{sec.1.}.
\begin{mthm}\label{main.thm.}
  Let $X_A(p)$ be the intersection 
  \begin{equation}
  \begin{cases}
	\sum_{i=1}^mA_i|z_i|^{p}=\bm{0}, \\
	\lVert z\rVert_{p}=1.
  \end{cases}\label{Def.X.}
\end{equation}
 There is a unique point $f_p(z)$ in the leaf $L_z$ for each 
  $z\in\mathcal{S}_A$, whose 
  $L^p$-norm $\lVert f_p(z)\rVert_p$ is minimal and positive,  
  and the restriction of the smooth function $f_2/\lVert f_2\rVert_2\co S_A\to X_A(2)$ 
  to $X_A(p)$ induces a homeomorphism onto $X_A(2)$, for all $p\geq 1$.
  Moreover,
  \begin{equation*}
   \begindc{\commdiag}[15]
   \obj(0,1)[aa]{$\Phi_A(p)\co X_A(p)
	 \times\mathbb{R}^d\times\mathbb{R}_{>0}$}
   \obj(7,1)[bb]{$\mathcal{S}_A$}
   \obj(0,0)[cc]{$(z,T,r)$}
   \obj(7,0)[dd]{$r(z_ie^{\langle A_i,T \rangle})_{i=1}^m$,}
   \mor{aa}{bb}{}
   \mor{cc}{dd}{}[+1,6]
   \enddc
 \end{equation*}
  is a homeomorphism.
\end{mthm}
Similar to \eqref{Def.T2.}, for each $p$ we can define a continuous 
function $T_p\co \mathcal{S}_A\to \mathbb{R}^d$ such that 
$f_p/\lVert f_p\rVert_p\co\mathcal{S}_A\to X_A(p)$ is a retraction, 
where $f_p(z)={F}(z,T_p(z))$ is the function of 
$L^p$-norm minima in the leaf $L_z$.

It is interesting to imagine what will happen when $p$ tends to
infinity, and this will be discussed in Section \ref{sec.2.}.
First note that the set 
  \begin{equation}
	K_A=\{\sigma\subset[m]\mid \bm{0}\in\mathrm{conv}A([m]\setminus\sigma)
	\}
	\label{Def.KA.}
  \end{equation}
  is an \emph{abstract simplicial complex} (cf.~\cite[Lemma 0.12]{BM06}), 
  i.e.~all subsets of $\sigma$ will be
  in $K_A$ if $\sigma$ is. It turns out that 
  with each $z\in\mathcal{S}_A$ fixed, $T_p(z)$ and
$f_p(z)/\lVert f_p(z) \rVert_p$ are continuous in $p\in[1,\infty)$ 
  (see Proposition \ref{conti.Tp.});
when $p$ goes to infinity, $f_p(z)/\lVert f_p(z)\rVert_p$ approaches
to the \emph{moment-angle complex} $(D^2,S^1)^{K_A}$ (see \eqref{M.A.C.} 
and Proposition \ref{go.out.} for details), which 
  is a subset of the intersection of $\mathcal{S}_A$ with 
  the $L^{\infty}$-norm unit sphere in $\mathbb{C}^m$
  ($\lVert z \rVert_{\infty}=\max\{|z_i|\}_{i=1}^{m}$).

We say that the tuple $A$ is \emph{centered at the origin}, 
if the centroid of all vectors in $A$ is located at the origin:  
 \begin{equation}
   \sum_{i=1}^{m}A_i=\bm{0}.
   \label{cent.orig.}
 \end{equation}
Under this additional assumption, $K_A$ is isomorphic 
to the boundary of a convex polytope arising from the 
\emph{Gale transform} of $A$ (see Proposition \ref{Lem.Of.Basis.});
based on a result of Panov and Ustinovsky in \cite{PU12},  
in Section \ref{sec.3.} we will show that 
$f_p(z)/\lVert f_p(z) \rVert_p$ converges to a unique 
point in $(D^2,S^1)^{K_A}$ as $p$ tends to infinity.
With a similar treatment as the one for Theorem \ref{main.thm.},  
the following theorem holds:
 \begin{mthm}\label{mthm2.}	  
	Assume that $A$ is an admissible tuple centered at the origin.
	Then the restriction $f_2/\lVert f_2\rVert_2|_{(D^2,S^1)^{K_A}}\co
  (D^2,S^1)^{K_A}\to X_A(2)$ 
  is a homeomorphism. Moreover,
  \begin{equation*}
   \begindc{\commdiag}[15]
   \obj(0,1)[aa]{$\Phi_A(\infty)\co (D^2,S^1)^{K_A}\times\mathbb{R}^d\times\mathbb{R}_{>0}$}
   \obj(9,1)[bb]{$\mathcal{S}_{A}$}
   \obj(0,0)[cc]{$(z,T,r)$}
   \obj(9,0)[dd]{$r(z_ie^{\langle A_i,T \rangle})_{i=1}^m$}
   \mor{aa}{bb}{}
   \mor{cc}{dd}{}[+1,6]
   \enddc
 \end{equation*}
  is a homeomorphism. 
\end{mthm}
Therefore, we can understand such a moment-angle complex
$(D^2,S^1)^{K_A}$ as ``$X_A(\infty)$'', namely 
the intersection of the $L^{\infty}$-norm minima in the Siegel leaves with the 
$L^{\infty}$-norm unit sphere in $\mathbb{C}^m$	
(the reader is encouraged to imagine the deformation from $X_A(1)$ to 
$X_A(\infty)$ in the case $d=0$).

As an application, in Section \ref{sec.4.} 
we give an alternative proof 
for a rigidity theorem of Bosio and Meersseman (cf.~\cite[Theorem 4.1]{BM06}):
if two admissible $m$-tuples $A$ and $A'$ are both centered at the origin 
such that $K_A$ and $K_{A'}$ are isomorphic simplicially,  
then there is a diffeomorphism between associated links $X_A(2)$ and $X_{A'}(2)$
(see Proposition \ref{rigi.} for more details).

From its definition \eqref{Def.F.}, notice that 
each leaf ${L}_z$ is contained in $\mathcal{S}_A\cap\mathbb{R}^m$ if 
and only if $z\in\mathcal{S}_A\cap\mathbb{R}^{m}$. Hence 
the theorems and properties above are also true when 
restricted to the subspace $\mathbb{R}^m$ in $\mathbb{C}^m$.

At last in Section \ref{sec.4.}, we shall illustrate 
that the restriction of $f_2/\lVert f_2\rVert_{2}$
to the real moment-angle complex 
$(D^1,S^0)^{K_A}=(D^2,S^1)^{K_A}\cap\mathbb{R}^m$ 
is a \emph{piecewise differentiable} homeomorphism 
onto $X_A(2)\cap\mathbb{R}^m$, provided that $A$ is 
admissible and centered at the origin
(see Definition \ref{smooth.PL.}, Lemma \ref{inj.lem.} 
and Proposition \ref{fina.pl.} for more details). In this way 
these real moment-angle complexes
can be smoothed, as piecewise linear manifolds.

\section{Proof of Theorem \ref{main.thm.}}\label{sec.1.}
We start with a well-known lemma due to Meersseman and Verjovsky,
whose proof is omitted here:
\begin{lem}[cf.~{\cite[Lemma 1.1]{MV04}}, {\cite[Lemma 0.3]{BM06}}]\label{adm.lem.}
  For an admissible tuple $A=(A_i)_{i=1}^m$, let 
  $\widetilde{A}=(\widetilde{A}_i)_{i=1}^m$ be the augmentation 
  with 
  $\widetilde{A}_i=(A_i^{\mathrm{T}},1)^{\mathrm{T}}\in\mathbb{R}^{d+1}$, 
  $i=1,2,\dots,m$. Then for any $I\subset[m]$ such that
  $\bm{0}\in\mathrm{conv}A(I)$, the rank of the subtuple $\widetilde{A}(I)$ 
  is $d+1$.
\end{lem}
\begin{prop}\label{exi.uni.}
 For each $z\in\mathcal{S}_A$ given, there
 is a unique point $f_p(z)$ in the leaf $L_z$, such that 
 $\lVert f_p(z)\rVert_p$ is minimal and positive. 	   
 \end{prop}
\begin{proof}
 (Uniqueness (cf.~\cite{CKP78}, \cite{Mee00}, \cite{MV04}).) 
 Assume that $F_z$ has 
 two local minima, i.e.~$T_1$ and $T_2$ in $\mathbb{R}^d$ 
 such that they are both critical points 
 of 
 $(\lVert{F}(z,T)\rVert_p)^p=\sum_{i=1}^m|z_i|^pe^{p \langle A_i,T\rangle}$, 
 which means
  	 \[\sum_{i=1}^mA_i|z_i|^pe^{p\langle A_i,T_j\rangle}=\bm{0}, 
   \quad j=1,2.
   \]
   We define a function 
   $h\co [0,1]\to\mathbb{R}$ such that
   $h(t)=(\lVert{F}(z,(1-t)T_1+tT_2)\rVert_p)^p$,
  clearly 
   \begin{equation} 
	 \frac{\mathrm{d}h}{\mathrm{d}t}=p\sum_{i=1}^m\langle A_i,T_2-T_1 \rangle |z_i|^p
	e^{p \langle A_i,(1-t)T_1+tT_2\rangle}. \label{Diff.F.}
	\end{equation}
	From Lemma \ref{adm.lem.}, the subtuple $A(I_z)$ has rank $d$ 
	($I_z\subset [m]$ consists of entries $i$ such that 
  $z_i\not=0$), which is independent 
  of $z\in\mathcal{S}_A$, thus 
  there exists $i\in I_z$ such that $\langle A_i, T_2-T_1\rangle$ 
  does not vanish; it follows that
  the second derivative of $h$ is strictly positive,  
  hence its first derivative \eqref{Diff.F.} 
  is strictly increasing, which is a contradiction.\\ 
  (Existence.) First from Cauchy-Schwarz inequality 
  \begin{equation}
   \lVert F(z,T)\rVert_2\leq\lVert F(z,T)\rVert_1\leq\sqrt{m}\lVert F(z,T)\rVert_2,
	\label{Cauc.Schw.}
  \end{equation}
together with Theorem \ref{funda.thm.} and the Lemma \ref{Conc.Lem.} below, 
we conclude that $\lVert F(z,T)\rVert_1$ bounds away from zero, 
and stays large whenever $\lVert T\rVert_2$ is large. 
Thus the minimum of $\lVert F(z,T)\rVert_1$ is positive, 
and it appears only when $T$ is in the interior of a ball 
of finite radius.
So the case $p=1$ is clear. For general cases when $p\not=1,2$, 
H\"{o}lder's inequality implies
\begin{equation}
  \lVert F(z,T)\rVert_p\leq\lVert F(z,T)\rVert_1\leq
\sqrt[q]{m}\lVert F(z,T)\rVert_p,
  \label{Hold.}
\end{equation}
here $q>1$ such that $1/p+1/q=1$. We can 
repeat the previous argument and then the proof is completed.
\end{proof}

\begin{lem}\label{Conc.Lem.}
With $z\in\mathcal{S}_A$ given,	for any $N> 0$, there exists 
 $R>0$, such that $\lVert F(z,T)\rVert_2> N$ whenever
$\lVert T\rVert_2>R$.
\end{lem}
\begin{proof}
Let $T_2(z)$ be the point in $\mathbb{R}^d$ such that 
$\lVert{F}(z,T_2(z))\rVert_2$ is minimal (see \eqref{Def.T2.} for details).
 Denote by $u(t;T_1,T_2)$ the derivative 
of $(\lVert F(z,(1-t)T_1+tT_2)\rVert_2)^2$ with respect to 
$t\in [0,1]$, for 
$T_1,T_2\in \mathbb{R}^d$,
and let $B(r,T_2(z))$ be the ball 
with radius $r$ centered at $T_2(z)$. 
 Since $T_2(z)$ is the unique minimum, 
 for all $y\in \partial B(1,T_2(z))$ on the 
 boundary, 
 there is a positive $\varepsilon$, such that
\[
 (\lVert{F}(z,y)\rVert_2)^2-
 (\lVert{F}(z,T_2(z))\rVert_2)^2=
 \int_{0}^{1}u(t;T_2(z),y)dt >\varepsilon,
\]
therefore we can choose
  $t(y)\in(0,1)$ such that 
  \[ u(t(y);T_2(z),y)>\varepsilon,\]
  by the mean value theorem. 
  For $r>1$, assume that 
  $y_r\in \partial B(r,T_2(z))$ with $y\in\partial B(1,T_2(z))$
   on the ray from $T_2(z)$ to $y_r$, 
  by the monotonicity of $u(t;T_2,y_r)$ (see the uniqueness part in 
  the proof of Proposition \ref{exi.uni.}), 
  we have 
  \[u(t;y,y_r)>u(t(y);T_2,y),\]
  thus
  \begin{align*}
  (\lVert{F}(z,y_r)\rVert_2)^2-
  (\lVert{F}(z,T_2(z))\rVert_2)^2=
  \int_{0}^{1}u(t;T_2,y)dt+\int_{0}^{1}u(t;y,y_r)dt
  >\varepsilon+(r-1)\varepsilon,
  \end{align*}
 from which the conclusion follows.
\end{proof}

The function of minima 
$f_p\co\mathcal{S}_A\to \mathcal{S}_A$ is 
well defined by Proposition \ref{exi.uni.}; 
but except for the case $p=2$, it remains to prove its 
continuity. In what follows 
we shall illustrate this by showing 
the continuity of the restriction
$f_p/\lVert f_p\rVert_p\big|_{X_A(2)}$ first, 
and then it will follow from the global diffeomorphism $\Phi_A(2)$ 
defined in Theorem \ref{funda.thm.}.
\begin{prop}\label{conti.}
   The restriction $f_2/\lVert f_2\rVert_2\big|_{X_A(p)}$ 
   of the smooth function $f_2/\lVert f_2\rVert_2\co S_A\to X_A(2)$ 
 induces a homeomorphism onto $X_A(2)$, whose 
 inverse is $f_p/\lVert f_p\rVert_p\big|_{X_A(2)}\co X_A(2)\to X_A(p)$. 
\end{prop}
\begin{proof}
 Consider the function 
  \[
\begindc{\commdiag}[15]
\obj(0,1)[aa]{$\Phi_A\co\mathcal{S}_A
  \times\mathbb{R}^d\times\mathbb{R}_{>0}$}
\obj(8,1)[bb]{$\mathcal{S}_A$}
\obj(0,0)[cc]{$(z,T,r)$}
\obj(8,0)[dd]{$r(z_ie^{\langle A_i,T \rangle})_{i=1}^m$,}
\mor{aa}{bb}{}
\mor{cc}{dd}{}[+1,6]
\enddc
  \]
from Theorem \ref{funda.thm.} and Proposition \ref{exi.uni.}, 
with given $z\in\mathcal{S}_A$, its image 
under $\Phi_A$ intersects both $X_A(p)$ and 
$X_A(2)$ exactly once, respectively,
hence $f_2/\lVert f_2\rVert_2\big|_{X_A(p)}$ is 
a bijection. Moreover, it is easy
to see that $X_A(p)$ is compact and $X_A(2)$
is Hausdorff; since a closed subspace of a compact space 
is compact, and a compact subspace of a Hausdorff space 
is closed, it follows that  
$f_2/\lVert f_2\rVert_2\big|_{X_A(p)}$
is closed 
and hence a homeomorphism by the bijectiveness.
As a conclusion, its inverse 
$f_p/\lVert f_p\rVert_p\big|_{X_A(2)}$ 
is continuous.
 \end{proof}
\begin{thm}\label{final.1.}
 The continuous function
\begin{equation*}
 \begindc{\commdiag}[15]
  \obj(0,1)[aa]{$\Phi_A(p)\co X_A(p)
    \times\mathbb{R}^d\times\mathbb{R}_{>0}$}
  \obj(8,1)[bb]{$\mathcal{S}_A$}
  \obj(0,0)[cc]{$(z,T,r)$}
  \obj(8,0)[dd]{$r(z_ie^{\langle A_i,T \rangle})_{i=1}^m$}
  \mor{aa}{bb}{}
  \mor{cc}{dd}{}[+1,6]
  \enddc
\end{equation*}
is a homeomorphism, for all $p\geq 1$. 	
\end{thm}
\begin{proof}
It suffices to find a continuous inverse for $\Phi_A(p)$.
Suppose that 
$f_p(z)/\lVert f_p(z)\rVert_p=(x_i(z))_{i=1}^m$; 
  for $(z,T,r)\in X_A(2)
  \times\mathbb{R}^d\times\mathbb{R}_{>0}$, we can rewrite 
  \[
	z=\rho^{-1}(z)F( f_p(z)/\lVert f_p(z)\rVert_p,T_2(f_p(z)/\lVert f_p(z)\rVert_p))=
 \rho^{-1}(z)(x_i(z)e^{\langle A_i, T_2(f_p(z)/\lVert f_p(z)\rVert_p)\rangle})_{i=1}^m, 
 \]
 where $\rho(z)=\lVert(x_i(z)e^{\langle A_i, T_2(f_p(z)/\lVert f_p(z)\rVert_p)\rangle}
 )_{i=1}^m\rVert_2$.
The continuity of $\rho^{-1}(z)$, $x_i(z)$ 
and $e^{\langle A_i, T_2(f_p(z)/\lVert f_p(z)\rVert_p)\rangle}$ 
 follows from
 Proposition \ref{conti.} (by Theorem \ref{funda.thm.}, $T_2$ 
 is smooth). 
 Observe that 
	\begin{align*}
	 \Phi_A(2)(z,T,r)&=r(z_ie^{\langle A_i,T \rangle})_{i=1}^m
	=r\rho^{-1}(z)(x_i(z)e^{\langle A_i, T+T_2(f_p(z)/
	\lVert f_p(z)\rVert_p)\rangle})_{i=1}^m\\
	&=\Phi_A(p)(f_p(z)/\lVert f_p(z)\rVert_p,T+T_2(f_p(z)/
	\lVert f_p(z)\rVert_p),r\rho^{-1}(z)),
	\end{align*}
hence we have a coordinate transition function
  \begin{equation*}
   \begindc{\commdiag}[5]
   \obj(0,3)[aa]{$\varphi\co X_A(2)
	 \times\mathbb{R}^d\times\mathbb{R}_{>0}$}
   \obj(40,3)[bb]{$X_A(p)
	 \times\mathbb{R}^d\times\mathbb{R}_{>0}$}
   \obj(0,0)[cc]{$(z,T,r)$}
   \obj(40,0)[dd]{$(f_p(z)/\lVert f_p(z)\rVert_p,T+T_2(f_p(z)/
	                 \lVert f_p(z)\rVert_p),r\rho^{-1}(z))$.}
   \mor{aa}{bb}{}
   \mor{cc}{dd}{}[+1,6]
   \enddc
 \end{equation*}
 It is straightforward to check the continuity of $\varphi$, 
 thus $\varphi\circ(\Phi_A(2))^{-1}$ is the inverse of $\Phi_A(p)$.
\end{proof}
 \begin{cor}\label{solu.}The function
\begin{equation}
T_p\co\mathcal{S}_A\to \mathbb{R}^d, \text{ s.t. } f_p(z)=F(z,T_p(z)),
\label{Def.Tp.}
\end{equation}
is well defined and continuous.  
That is to say, for each $z\in\mathcal{S}_A$,
$T_p(z)$ is the unique solution of the equation
\[\sum_{i=1}^m A_i|z|_i^pe^{p\langle A_i, T \rangle}=\bm{0},
\]
which depends on $z$ continuously. 
\end{cor}

\section{When $p$ tends to infinity}\label{sec.2.}
In this section we treat $T_p(z)$ and $f_p(z)$ that are 
defined in Corollary \ref{solu.} and Proposition \ref{exi.uni.} respectively
as functions of $p\in[1,\infty)$, with 
$z\in\mathcal{S}_A$ fixed. 
\begin{lem}\label{bound.lem.}
	There exists a bound $N(z)$, such that 
	$\lVert T_p(z)\rVert_2<N(z)$ for all $p\in[1,\infty)$.
\end{lem}
\begin{proof}
By definition, $\lVert {F}(z,T_p(z))\rVert_p$ is 
the unique minimum in the leaf $L_z$. 
Suppose that on the contrary, 
there exists a sequence $\{p_k\}_{k=1}^{\infty}$ tending to 
infinity, such that $\lVert T_{p_k}(z)\rVert_2>k$, for all 
$k$. 
First by Lemma \ref{Conc.Lem.} and Cauchy-Schwarz
inequality \eqref{Cauc.Schw.}, $\lVert F(z,T)\rVert_1$
becomes arbitrary large whenever $\lVert T\rVert_2$ is 
large enough, thus 
\[\exists N>0, \ s.t.\ \forall k>N, \ 
 m \lVert F(z,T_1(z))\rVert_1<\lVert F(z,T_{p_k}(z))\rVert_1.
  \]
Then by H\"{o}lder's inequality \eqref{Hold.}, we have
\[
  \sqrt[q_k]{m}\lVert F(z,T_1(z))\rVert_{p_k}\leq
  \sqrt[q_k]{m}\lVert F(z,T_1(z))\rVert_1<\lVert F(z,T_{p_k}(z))\rVert_1
  \leq\sqrt[q_k]{m}\lVert F(z,T_{p_k}(z))\rVert_{p_k}
  \]
 where $1/p_k+1/q_k=1$, it follows that 
 $\lVert F(z,T_{p_k}(z))\rVert_{p_k}$ is strictly greater than 
 $\lVert F(z,T_1(z))\rVert_{p_k}$, 
 yielding a contradiction.
\end{proof}

\begin{prop}\label{conti.Tp.}
 $T_p(z)$ is continuous	for all $p\in[1,\infty)$.
\end{prop}
\begin{proof}
Suppose again on the contrary, there is a sequence 
$\{p_k\}_{k=1}^{\infty}$ with $\lim_{k}p_k=p_0$,
but $\lVert T_{p_k}(z)-T_{p_0}(z)\rVert_2\geq \delta$, for some
$\delta>0$. Without loss of generality we may assume that 
$\lim_{k}T_{p_k}=T_0\not=T_{p_0}(z)$, or we can choose a subsequence 
satisfying the property, by the lemma above. 
Consider the smooth function
  \begin{equation*}
    \begindc{\commdiag}[5]
    \obj(0,3)[aa]{$\mu\co[1,\infty)\times\mathbb{R}^d$}
 	 \obj(20,3)[bb]{$\mathbb{R}^d$}
    \obj(0,0)[cc]{$(p,T)$}
    \obj(20,0)[dd]{$\sum_{i=1}^mA_i|z_i|^{p}e^{p\langle A_i, T\rangle}$,}
    \mor{aa}{bb}{}
    \mor{cc}{dd}{}[+1,6]
    \enddc
  \end{equation*}
we have $\bm{0}=\lim_{k}\mu(p_k,T_k(z))=\mu(p_0,T_0)$ by continuity, contradicting
to the uniqueness (see Corollary \ref{solu.}). 
\end{proof}
\begin{cor}\label{To.The.Sphere.}
As a function of $p\in[1,\infty)$, 
$f_p(z)/\lVert f_p(z)\rVert_p$ is continuous with its image 
in the $L^{p}$-link $X_A(p)$ (defined by \eqref{Def.X.}), and we have
\begin{equation}
  \lim_{p\to\infty}\lVert f_p(z)/\lVert f_p(z) \rVert_p\rVert_{\infty}=1.
  \label{Appro.}
\end{equation}
\end{cor}
\begin{proof}
Denote $f_p(z)/\lVert f_p(z) \rVert_p$ by $y(p)=(y_i(p))_{i=1}^{m}$. 
Observe that 
\[1=\lVert y(p) \rVert_p=\lVert y(p)\rVert_{\infty}
  \left( \sum_{i=1}^m\Big|\frac{y_i(p)}{\lVert y(p)
    \rVert_{\infty}}\Big|^{p} \right)^{\frac{1}{p}},\]
where the last term in the bracket does not exceed $m$,
thus \eqref{Appro.} holds as desired.
\end{proof}
     
\subsection{Moment-angle complexes}\label{S.M.A.C.}
Let $K_A$ be the simplicial complex 
defined by \eqref{Def.KA.}. 
The associated moment-angle complex $(D^2,S^1)^{K_A}$ is 
defined as the polyhedral product
\begin{equation}
	 (D^2,S^1)^{K_A}=\bigcup_{\sigma\in K_{A}}D(\sigma); \ D(\sigma)=\prod_{i=1}^mY_i, 
  \  Y_i=\begin{cases}D^2=\{|z|\leq 1\mid z\in\mathbb{C}\} & \text{ if $i\in \sigma$,}\\
	S^1=\{|z|=1\mid z\in\mathbb{C}\}  &  \text{otherwise.}
  \end{cases}\label{M.A.C.}
\end{equation}

The proposition below implies 
that $f_p(z)/\lVert f_p(z) \rVert_p\in X_A(p)$
approaches $(D^2,S^1)^{K_A}$ 
as $p$ tends to infinity.
\begin{prop}\label{go.out.}
Let $S_{\infty}$ be the unit sphere 
of $\mathbb{C}^m$ with respect to the 
$L^{\infty}$-norm, and let $z\in\mathcal{S}_A$ be 
a given point.
Then for every point
$z'=(z_i')_{i=1}^m\in S_{\infty}\cap\mathcal{S}_{A}\setminus(D^2,S^1)^{K_A}$, 
$f_p(z)/\lVert f_p(z)\rVert_p$ will go
outside of the set
\[
  C(z')=\{(z_i)_{i=1}^m\in\mathbb{C}^m\mid 
	|z_i|\leq |z_i'|,\forall i=1,2,\dots,m\},
\]
whenever $p$ is sufficiently large.
\end{prop}
\begin{proof}
Denote by $B\subset\mathbb{R}^d$ the union of 
all convex hulls of the form $\mathrm{conv}A([m]\setminus\tau)$ with $\tau\subset [m]$
not contained in $K_A$ (in other words, 
$\bm{0}\not\in\mathrm{conv}A([m]\setminus\tau)$). 
It is clear that $B$ is empty when and only when
$K_A$ bounds the $(m-1)$-simplex (i.e.~$K_A=2^{[m]}\setminus[m]$, this 
happens only when $d=0$, by the admissibility of $A$), 
which means $S_{\infty}=(D^2,S^1)^{K_A}$ 
and we have nothing to prove;
otherwise $B$ is compact
thus there is an open neighborhood $U_B$ such that $\bm{0}\not\in U_B$.

Suppose the contrary, namely there is a 
sequence $\{p_k\}_{k=1}^{\infty}$ tending to infinity, 
such that $x_k=(x_{ki})_{i=1}^m=
f_{p_k}(z)/\lVert f_{p_k}(z)\rVert_p\in C(z')$. Since
$C(z')$ is compact, we may assume that $\{x_k\}_{k=1}^{\infty}$ converges 
to a point $x_0=(x_{0i})_{i=1}^m\in C(z')$, without loss of generality. 

We claim that the vector 
 \begin{equation}
   \sum_{i=1}^mA_i|x_{ki}|^{p_k}
   \label{conv.zero.}
 \end{equation}
lies in $U_B$ whenever $k$ is large enough. Notice that 
this will be a contradiction since $x_k\in X_{A}(p)$, whose definition
implies that the vector above should always be zero.

To see this, first note that because 
$x_0\not\in (D^2,S^1)^{K_A}$, there exists $\tau\not\in K_A$, 
such that $|x_{0i}|<\delta<1$ for all
$i\in \tau$. This means for those $i\in \tau$, 
there exists an $N>0$, such that  
$|x_{ki}|<\delta<1$ holds when $k>N$; thus 
for any given $\varepsilon> 0$, 
we can find $N_{\varepsilon}>N$, such that 
$|x_{ki}|^{p_k}<\varepsilon$ for all $k> N_{\varepsilon}$. 
It is not difficult to see that, 
if $\varepsilon$ is small enough, 
vector \eqref{conv.zero.} shall lie in $U_B$, 
as claimed. 
\end{proof}

\section{The convergence}\label{sec.3.}
In this section we shall prove that the function 
$f_p(z)/\lVert f_p(z)\rVert\co\mathcal{S}_A\to X_A(p)$
indeed converges to a point in $(D^2,S^1)^{K_A}$, 
as one may expect from Proposition \ref{go.out.}, with an additional
assumption that $A$ is centered at the origin (see \eqref{cent.orig.}).
The main technique we use here is combinatorial, in which 
Gale transforms play an essential role.\footnote{I would
like to thank the referee for pointing out that analogues of 
Lemma \ref{key.lem.} and Proposition \ref{Lem.Of.Basis.} are already proven
in \cite{MV04} and \cite{BM06}, where Gale transforms have 
been intensely used. The approach here is motivated by 
those in these works.}

Suppose $V=(V_1,V_2,\dots,V_m)$ is a tuple of vectors in 
$\mathbb{R}^{m-d-1}$, such that the affine dimension 
of $V$ is $m-d-1$, i.e. the matrix with columns 
$(V_i^{\mathrm{T}},1)^{\mathrm{T}}$ ($i=1,2,\dots,m$) 
has rank $m-d$.

Denote by $A_{V}=(A_1,A_2,\dots,A_m)$ the \emph{Gale transform} of $V$ 
(cf.~\cite[Chapter 5.4, pp.~85--86]{Gru03}), which is the transpose 
of a basis of solutions 
of the following linear system 
\begin{equation}
 \begin{cases}
   \sum_{i=1}^mV_ix_i=0,\\
   \sum_{i=1}^mx_i=0.
 \end{cases}\label{Gale.Trans.}
\end{equation}
It is clear that each $A_i$ is a vector in $\mathbb{R}^{d}$ and
different choices of $A_V$ are linearly equivalent.

Recall that for any $J \subset [m]$, the subtuple $V(J)=(V_i)_{i\in J}$ 
is a \emph{face} of $V$, if the intersection of $\mathrm{conv}V([m]\setminus J)$ with 
the affine space spanned by vectors in $V(J)$ is empty 
(cf.~\cite[Chapter 5.4, p.~88]{Gru03}). For instance, 
if $V$ consists of the vertices of a convex polytope $P$, then $V(J)$ is 
a face of $V$ when and only when $\mathrm{conv}V(J)$ is a face of $P$.
 Now we need two facts about Gale transforms:
 \begin{prop}[cf.~{\cite[Chapter 5.4, p.~88]{Gru03}}]\label{Funda.Gale.}
Let $V=(V_i)_{i=1}^m$ be a tuple of vectors in $\mathbb{R}^{m-d-1}$, 
  whose affine dimension is $m-d-1$, and let $A_V=(A_i)_{i=1}^m$ be its
  Gale transform.
  
Then for any 
$I\subset [m]$, $\mathrm{conv}V([m]\setminus I)$ is a face of 
$V$ if and only if either $I$ is empty or $\bm{0}$ is in the relative interior of 
$\mathrm{conv}A_{V}(I)$. 

Moreover, $V$ coincides with the vertex set
 of a convex polytope $P$ if and only if either 
\begin{enumerate}
 \item [(i)] $d=0$ (thus $P$ is a simplex) or 
 \item [(ii)] for every open halfspace $H^{+}$ of $\mathbb{R}^d$ containing
  $\bm{0}$ in its closure, we have 
  $\mathrm{card}(\{i\mid A_i\in H^{+}\})\geq 2$.
\end{enumerate}
\end{prop}
It follows that if $V$ is centered at the
origin, then the double Gale transform of $V$ gives the 
same configuration in $\mathbb{R}^{m-d-1}$. However, 
this is not true in general (see Remark \ref{rem:counter}). 

Based on the facts above, we have the following lemma (in which
we use the same notations with those in Proposition \ref{Funda.Gale.}).  
\begin{lem}\label{key.lem.}
Suppose that every vector of $V$ is a face, and
every face of $V$ has at most $m-d-1$ vectors. 
Then  $V$ coincides with the vertex set of 
a convex polytope $P$, whose boundary is simplicial.  
\end{lem}
\begin{proof}
Let $A_V=(A_i)_{i=1}^m$ be the Gale transform of $V$. 
It suffices to show either 
(i) or (ii) in Proposition \ref{Funda.Gale.} holds. Note that 
the case $d=0$ is trivial: this happens if and only if 
$V$ spans an $(m-1)$-simplex in $\mathbb{R}^{m-1}$.

Now suppose that $d>0$. Notice that by Proposition \ref{Funda.Gale.},
the Siegel and weak hyperbolicity conditions hold for $A_V$,
with the assumption above.
Moreover, since every vector in $V$ is a 
face, we have $\bm{0}\in\mathrm{conv}A_V(J)$, for all 
$J$ with $\mathrm{card}(J)=m-1$. 

Let $H^+$ be an open 
halfspace of $\mathbb{R}^{d}$ with $\bm{0}$ on the boundary. 
From its admissibility, $A_V$ has rank $d$,
with a neighborhood of $\bm{0}\in\mathbb{R}^d$ contained in 
$\mathrm{conv}A_V$ (see Lemma \ref{adm.lem.}), 
hence there exists $A_i\in A_V$, such that $A_i\in H^{+}$. 
Observe that now 
$\bm{0}\in\mathrm{conv}A_V([m]\setminus\{i\})$ with 
$A_V([m]\setminus\{i\})$ again admissible, by the same argument, 
there exists another $A_j\in A_V$ with $A_j\in H^{+}$, 
which means (ii) holds hence 
the statement follows.
\end{proof}
\begin{prop}\label{Lem.Of.Basis.}
Let $K_A$ be the simplicial complex induced from
an admissible $m$-tuple $A=(A_i)_{i=1}^m$ centered at the origin, 
with vectors in $\mathbb{R}^d$.
Let the tuple $V=(V_i)_{i=1}^m$ 
be the transpose of a basis of the system
\begin{equation}
  \begin{cases}
	\sum_{i=1}^mA_ix_i=\bm{0},\\
	\sum_{i=1}^mx_i=0,
  \end{cases}\label{Gale.Trans.Inv.}
\end{equation}
then $\{V_i\}_{\{i\}\in K_A}$ is the vertex set of a convex polytope $P_A$
of affine dimension $m-d-1$, with each $V_j$ in its interior, 
where $\{j\}\not\in K_A$. Moreover, the boundary of $P_A$ is isomorphic 
to $K_A$ and we can assume that $P_A$ contains $\bm{0}$ in its interior.
\end{prop}
\begin{proof}
First from Lemma \ref{adm.lem.}, the affine dimension of $V$ 
is $m-d-1$.

Since the centroid of $A$ is $\bm{0}$, 
now $A$ is the Gale transform of $V$,  
with the subtuple $(V_i)_{\{i\}\in K_A}$ satisfying
the assumptions in Lemma \ref{key.lem.}, thus it coincides
with the vertex set of a convex polytope $P_A$, whose boundary is simplicial. 
For those $\{j\}\not\in K_A$, if $V_j$ lies outside, or 
on the boundary of $P_A$, it must be in a face of $V$
that is contained in a supporting hyperplane of $P_A$; by 
Proposition \ref{Funda.Gale.}, this is impossible.

The last statement is also a consequence of Proposition \ref{Funda.Gale.},
together with the observation that any translation of the form 
$V+v_0=(V_i+v_0)_{i=1}^m$ also satisfies \eqref{Gale.Trans.Inv.}. 
\end{proof}
        
\begin{exm}
Let $A$ be the $5$-tuple given by the matrix
\[
  \begin{pmatrix}
	0 & 0 & 1 & 1 & -2\\
	1 & 1/2 & 0 & 0 & -3/2
  \end{pmatrix}
\]
which is admissible and centered at the origin. 
By solving \eqref{Gale.Trans.Inv.} 
we can choose $V$ that is given by
\[
  \begin{pmatrix}
	0 & 0 & -1 & 1 & 0\\
	6 & -9 & 2 & 0 & 1
  \end{pmatrix}.
\]
Observe that the last point $(0,1)^{\mathrm{T}}$ 
is in the interior of the square spanned by the other four 
vertices.
\end{exm}
\begin{remark}\label{rem:counter}
 Note that Proposition \ref{Lem.Of.Basis.} is independent
 of the choice of $V$.  
  If the centroid of $A$ is not at the origin, 
  Proposition \ref{Lem.Of.Basis.} may not hold. Consider the 
 case that $A$ is given by the matrix 
 (one can check its admissibility)
\[
  \begin{pmatrix}
	1 & 1  & 4 & -2 \\
	4 & -2 & 1 & 1
  \end{pmatrix},
\]
then we choose $V=(-1,-1,1,1)$ by \eqref{Gale.Trans.Inv.}, 
but now points $V_2=(-1)$ and $V_4=(1)$ are not contained
in the interior of $P_A=\mathrm{conv}(V_1,V_3)$.
This is because the Gale transform of $V$ can be
\[
  \begin{pmatrix}
 	0 & 0  & 1 & -1 \\
 	1 & -1 & 0 & 0
   \end{pmatrix},
\]
which is no longer admissible. 
\end{remark}
The following proposition is essentially due to Panov 
and Ustinovsky (cf.~\cite{PU12}).
\begin{prop}\label{Panov.Ustinovsky.}
Let $A=(A_i)_{i=1}^m$ be an admissible $m$-tuple centered at the origin.
Then for each $z\in\mathcal{S}_{A}$ given, there is 
a unique pair $(r,T)\in\mathbb{R}_{>0}\times\mathbb{R}^d$ 
such that 
$\Phi_{A}(z,T,r)=r(ze^{\langle A_i,T\rangle})_{i=1}^{m}\in (D^2,S^1)^{K_A}$
(see \eqref{M.A.C.} for definition).
\end{prop}
\begin{proof}
The proof presented here is adapted from 
\cite[Theorem 9.2, pp.~37--40]{Pan13}. Observe that 
in the trivial case when $d=0$, i.e.~$K_A$ is a simplex, 
we can simply take 
$r=\lVert z\rVert_{\infty}^{-1}$. 
In what follows suppose $d>0$.

Let $\mathbb{R}_{\geq 0}$ (resp.~$\mathbb{R}_{\leq 0}$) 
be the set of non-negative (resp.~non-positive) real 
numbers. Note that it suffices to prove the 
cases when $z\in(\mathbb{R}_{\geq 0})^{m}$, since for each 
$z=(z_i)_{i=1}^m\in\mathbb{C}^{m}$, 
there is a rotation  
$e^{\sqrt{-1}\theta}=(e^{\sqrt{-1}\theta_{i}})_{i=1}^{m}\in(S^1)^m$,
such that $e^{\sqrt{-1}\theta}z=(e^{\sqrt{-1}\theta_i}z_i)_{i=1}^m
\in(\mathbb{R}_{\geq 0})^{m}$, and we have
\[
  e^{\sqrt{-1}\theta}\Phi_{A}(z,T,r)=
     \Phi_{A}(e^{\sqrt{-1}\theta}z,T,r).
\]

For the tuple $A=(A_i)_{i=1}^{m}$, let 
$V=(V_i)_{i=1}^m$ be the tuple defined
in Proposition \ref{Lem.Of.Basis.}, which satisfies \eqref{Gale.Trans.Inv.}. 
Since $A$ is centered at the origin, the row vectors 
of $\widetilde{V}=(\widetilde{V}_i)_{i=1}^m$ with 
$\widetilde{V}_i=(V_i^\mathrm{T},1)^\mathrm{T}$ is a basis of 
the orthogonal complement of the space spanned by the row vectors
of $A$.

Let $\alpha$ be the linear morphism 
\begin{equation*}
  \begindc{\commdiag}[15]
  \obj(0,1)[aa]{$\alpha\co\mathbb{R}^{m}$}
  \obj(7,1)[bb]{$\mathbb{R}^{m-d}$}
  \obj(0,0)[cc]{$(x_i)_{i=1}^m$}
  \obj(7,0)[dd]{$\sum_{i=1}^m\widetilde{V}_ix_i$.}
  \mor{aa}{bb}{}
  \mor{cc}{dd}{}[+1,6]
  \enddc
\end{equation*}
For $x=(x_i)_{i=1}^{m}\in(\mathbb{R}_{\geq 0})^{m}$, we shall 
abbreviate $(\ln(x_i))_{i=1}^m$ as $\ln(x)$ in what follows. 

First we consider the case 
$z\in\mathcal{S}_{A}\cap (\mathbb{R}_{>0})^{m}$. 
Observe that there exists 
a pair $(r,T)\in\mathbb{R}_{>0}\times\mathbb{R}^d$ such that 
$y=(y_i)_{i=1}^m=\Phi_{A}(z,T,r)$ when and only when 
$\ln(y)-w-\ln(z)=(\langle A_i,T\rangle)_{i=1}^m$,
where $w=(w_i)_{i=1}^m$ with $w_i\equiv \ln(r)$, 
and this happens
if and only if the vector $\ln(y)-w-\ln(z)$ belongs to  
$\mathrm{Ker}(\alpha)$.

Let $(\mathbb{R}_{\leq 0},0)^{K_A}$ be the following polyhedral product
\begin{equation}
 (\mathbb{R}_{\leq 0},0)^{K_A}=\bigcup_{\sigma\in K_A}D(\sigma); 
	 \quad D(\sigma)=\prod_{i=1}^mY_i, 
  \  Y_i=\begin{cases}\mathbb{R}_{\leq 0} & \text{ if $i\in \sigma$,}\\
	\{0\}  &  \text{otherwise,}
  \end{cases}\label{poly.prod.2.}
\end{equation}
and it is clear that $y\in (D^2,S^1)^{K_A}\cap(\mathbb{R}_{>0})^{m}$ 
if and only if $\ln(y)\in(\mathbb{R}_{\leq 0},0)^{K_A}$, 
hence now it suffices to find a unique pair $(u,c)\in
\left( (\mathbb{R}_{\leq 0},0)^{K_A},\mathbb{R} \right)$,
such that
\begin{equation}
  \sum_{i=1}^m({V}_i^{\mathrm{T}},1)^{\mathrm{T}}(u_i+c)=\sum_{i=1}^m
  ({V}_i^{\mathrm{T}},1)^{\mathrm{T}}\ln(z_i)
  \label{Fate.}
\end{equation}
holds, where $u=(u_i)_{i=1}^m$. 
Let $\overline{P}_{A}$ be the convex polytope spanned by $\{-V_i\}_{\{i\}\in K_A}$. 
By Proposition \ref{Lem.Of.Basis.}, $\overline{P}_{A}$ contains a neighborhood 
of $\bm{0}$ in its interior, and the 
boundary of $\overline{P}_{A}$ is the union 
$-\bigcup_{\sigma\in K_A}\mathrm{conv}V(\sigma)$, 
which is simplicially isomorphic to $K_A$.
Therefore every vector $\nu$ in $\mathbb{R}^{m-d-1}$
has a unique expression $\rho\nu_0$, where 
$\rho\in\mathbb{R}_{\geq 0}$
and $\nu_0$ lies in the relative 
interior of the corresponding face. 
Together with the observation 
$\sum_{i=1}^{m}V_i=\bm{0}$ (see \eqref{Gale.Trans.Inv.}), 
we conclude that there exists a pair 
$(u,c)\in\left( (\mathbb{R}_{\leq 0},0)^{K_A},\mathbb{R} \right)$
such that 
\[ 
  \sum_{i=1}^{m}V_iu_i=\sum_{\{i\}\in K_A}V_iu_i=\sum_{i=1}^mV_i\ln(z_i),
  \quad \sum_{i=1}^{m}\left( \ln(z_i)-u_i \right)=\sum_{i=1}^mc=mc,
\] 
namely \eqref{Fate.} holds, which is unique by the construction. 

Next we consider general cases when
$z\in\mathcal{S}_{A}\cap\mathbb{R}^m_{\geq 0}$ 
with $\bar{I}_z=\{i\mid z_i=0\}$ not empty. 
First note that by definition, 
$\bar{I}_z$ is a simplex of $K_A$.
Let $\pi_z\co\mathbb{R}^{m-d-1}\to\mathbb{R}^{m-d-1-\mathrm{card}(\bar{I}_z)}$ 
be the orthogonal projection onto the linear subspace 
\[
  \bigcap_{i\in \bar{I}_z}
  \{\nu\in\mathbb{R}^{m-d-1}\mid\langle \nu, V_{i}\rangle=0\},
\]
and denote by $\mathrm{Link}(\bar{I}_z,K_A)$ the union
\[\{\sigma\in K_A\mid (\sigma\cup \bar{I}_z)\in K_A,\  
\sigma\cap \bar{I}_z=\emptyset\},\]
which is a subcomplex of 
$\mathrm{Star}(\bar{I}_z,K_A)=\{\sigma\in K_A\mid \bar{I}_z\subset \sigma\}$.
It is not difficult to see that in the image of $\pi_z$,
$\pi_z(\mathrm{conv}V(\mathrm{Star}(\bar{I}_z,K_A))$ is 
a convex polytope bounded by 
$\pi_z(\mathrm{conv}V(\mathrm{Link}(\bar{I}_z,K_A))$ 
(for example, by induction on $\mathrm{card}(\bar{I}_z)$). 
Then by a similar argument as in the previous case, 
we deduce that there exists a unique
$u=(u_i)_{i=1}^m$ in the polyhedral product 
$(\mathbb{R}_{\leq 0},0)^{\mathrm{Link}(\bar{I}_z,K_A)}$ (defined
by replacing $K_A$ with $\mathrm{Link}(\bar{I}_z,K_A)$ in \eqref{poly.prod.2.}), 
such that 
\[
  \pi_z(\sum_{i=1}^mV_iu_i)=
  \pi_z(\sum_{i=1}^mV_i\chi(z_i)\ln(z_i)); 
  \ \ \chi(z_i)\ln(z_i)=\begin{cases}
	\ln(z_i)  &\text{if $|z_i|>0$,} \\
	0       & \text{otherwise};
  \end{cases}
\]
note that vectors of $\{V_i\}_{i\in\bar{I}_z}$ are
linearly independent, hence we have a unique 
$x=(x_i)_{i=1}^m\in\mathbb{R}^m$ 
with $I_{x}\subset \bar{I}_z$, 
such that 
\[
	\sum_{i=1}^m{V}_i(u_i+x_i)=
	\sum_{i=1}^m{V}_i\chi(z_i)\ln(z_i)
\]
holds. With $c$ obtained from
\[\sum_{i=1}^m(\chi(z_i)\ln(z_i)-u_i-x_i)=mc,\]
we have 
\[
 \sum_{i=1}^m({V}_i^{\mathrm{T}},1)^{\mathrm{T}}(u_i+x_i+c)=
 \sum_{i=1}^m({V}_i^{\mathrm{T}},1)^{\mathrm{T}}\chi(z_i)\ln(z_i).
\] 

At last, by solving 
$T\in\mathbb{R}^d$ from 
\[\langle A_i, T\rangle=\chi(z_i)\ln(z_i)-u_i-x_i-c\] 
for $i=1,2,\dots,m$, and setting $r=e^{c}$, we have 
$\Phi_{A}(z,T,r)\in (D^2,S^1)^{K_A}$ as desired; 
the uniqueness follows from the arguments above and 
the observation that the rank of $A$ is $d$.
\end{proof}

From Proposition \ref{Panov.Ustinovsky.},
we can define a map
$f_{\infty}\co\mathcal{S}_A\to \mathcal{S}_A$,
with $f_{\infty}(z)$ the point in the leaf $L_z$ such that  
$f_{\infty}(z)/\lVert f_{\infty}(z) \rVert_{\infty}\in (D^2,S^1)^{K_A}$. 

The proofs of the following Proposition \ref{conti.sec.} 
and Theorem \ref{final.2.} are 
similar to the ones for Proposition \ref{conti.} and 
Theorem \ref{final.1.}, respectively,
which we shall omit here. 
\begin{prop}\label{conti.sec.}
With the assumption that $A$ is admissible and centered at the origin,
the restriction $f_2/\lVert f_2\rVert_2\big|_{(D^2,S^1)^{K_A}}
\co(D^2,S^1)^{K_A}\to X_A(2)$ 
is a homeomorphism, whose 
inverse is the restriction 
$f_{\infty}/\lVert f_{\infty}\rVert_\infty\big|_{X_A(2)}
\co X_A(2)\to(D^2,S^1)^{K_A}$.
\end{prop}
\begin{thm}\label{final.2.}
The continuous function
\begin{equation*}
 \begindc{\commdiag}[15]
 \obj(0,1)[aa]{$\Phi_A(\infty)\co (D^2,S^1)^{K_A}
 \times\mathbb{R}^d\times\mathbb{R}_{>0}$}
 \obj(8,1)[bb]{$\mathcal{S}_A$}
 \obj(0,0)[cc]{$(z,T,r)$}
 \obj(8,0)[dd]{$r(z_ie^{\langle A_i,T \rangle})_{i=1}^m$}
 \mor{aa}{bb}{}
 \mor{cc}{dd}{}[+1,6]
 \enddc
\end{equation*}
is a homeomorphism, provided that 
$A$ is admissible and centered at the origin.
\end{thm}
Recall that for each $p\in [1,\infty)$, we have defined 
$T_p\co \mathcal{S}_A\to\mathbb{R}^d$ such that 
$f_p(z)=F(z,T_p(z))$ has the minimal $L^p$-norm
in each leaf $F_z$. By Theorem \ref{final.1.}, 
$T_p$ is the composition of $\Phi_A^{-1}(p)$ and 
the projection onto $\mathbb{R}^d$, and 
$f_{p}(z)/\lVert f_p(z)\rVert_{p}$ 
  is the composition of $\Phi_{A}^{-1}(p)$ and the 
projection onto $X_A(p)$. 
\begin{cor}
 Let $T_{\infty}\co\mathcal{S}_A\to\mathbb{R}^d$ be the 
composition of $\Phi_{A}^{-1}(\infty)$ and the 
projection onto $\mathbb{R}^{d}$, with $A$ admissible and 
centered at the origin. Then we have
\[\lim_{p\to\infty}T_p(z)=T_{\infty}(z), \] 
  which means 
\[ 
  \lim_{p\to\infty}f_{p}(z)/\lVert f_p(z)\rVert_{p}=
 f_{\infty}(z)/\lVert f_\infty(z)\rVert_{\infty},
\]
 with any $z\in\mathcal{S}_A$ given.
\end{cor}
\begin{proof}
By  Lemma \ref{bound.lem.}, Corollary \ref{To.The.Sphere.} and Proposition \ref{go.out.}, 
there exists a sequence $\{p_k\}_{k=1}^{\infty}$ such 
that $\{T_{p_k}(z)\}_{k=1}^{\infty}$ converges to a point $T_0\in\mathbb{R}^d$, 
and 
$\{f_{p_k}(z)\}_{k=1}^{\infty}$ converges to some $y_0$ 
such that $y_0/\lVert y_0\rVert_{\infty}\in(D^2,S^1)^{K_A}$. 
We claim that 
\begin{equation*}
\lim_{p\to\infty}T_p(z)=T_0=T_{\infty}(z) 
\end{equation*}
 with 
\begin{equation*}
 \lim_{p\to\infty}f_{p}(z)/\lVert f_p(z)\rVert_{p}=
y_0/\lVert y_0\rVert_{\infty}=f_{\infty}(z)/\lVert f_\infty(z)\rVert_{\infty}. 
\end{equation*}
Note that 
\[y_0/\lVert y_0\rVert_{\infty}=
\lim_{k\to\infty}\Phi_A(z,T_k(z),\lVert f_{p_k}\rVert_{p_k}^{-1})=
\Phi_A(z,T_0,\lVert y_0\rVert_{\infty}^{-1})\in (D^2,S^1)^{K_A}\]
which is uniquely determined by $z$ (see Proposition \ref{Panov.Ustinovsky.}), 
therefore
$y_0/\lVert y_0\rVert_{\infty}$ must 
be $f_{\infty}(z)/\lVert f_\infty(z)\rVert_{\infty}$ 
and $T_0$ must be $T_{\infty}(z)$. It is not difficult 
to see that the argument above is independent of the choice
of the sequence $\{p_k\}_{k=1}^{\infty}$, hence the claim holds and the 
proof is completed.
\end{proof}

\section{Some applications}\label{sec.4.}
In this section we shall revisit several known 
results from another perspective. 
First notice that by Proposition \ref{Lem.Of.Basis.}, a simplicial complex 
$K_A$ induced from an admissible tuple that is centered at the origin 
can be realized as the 
boundary of a convex polytope dual to a simple one;
the converse is also true: for a convex 
polytope with simplicial boundary, the Gale transform of its 
vertices will be a tuple with the property above.

Our first application is an alternative proof of a rigidity theorem
on polytopal moment-angle manifolds, due to Bosio and Meersseman:
\begin{prop}[cf.~{\cite[Theorem 4.1]{BM06}}]\label{rigi.}
 Let $K_A$ and $K_{A'}$ 
 be the simplicial complexes induced from two admissible $m$-tuples
 $A$ and $A'$ that are centered at the origin, respectively.
 If there is a simplicial isomorphism 
 $\phi\co K_A\to K_A'$, then there is a diffeomorphism 
 between $X_A(2)$ and $X_{A'}(2)$.
\end{prop}
\begin{proof}
 Observe that under the assumption, $\phi$ can be extended
 as a bijection from $[m]$ to itself (possibly not unique),
 and let $\widetilde{\phi}\co \mathcal{S}_A\to\mathcal{S}_{A'}$ 
 be the diffeomorphism via permuting coordinates 
 with respect to $\phi$. Clearly $\widetilde{\phi}$ 
 gives a homeomorphism between associated moment-angle complexes 
 $(D^2,S^1)^{K_A}$ and $(D^2,S^1)^{K_{A'}}$.
 On the other hand, we have a smooth map 
 $(f_2'/\lVert f_2'\rVert_2)\circ\widetilde{\phi}\co X_A(2)\to 
 X_{A'}(2)$ given in the diagram
 \[
 \begin{CD}
 \mathcal{S}_A @>\widetilde{\phi}>\mathrm{diffeo.}>\mathcal{S}_{A'}\\
 @AAA                                      @V f_2'/\lVert f_2'\rVert_{2} VV\\
   X_A(2)@>(f_2'/\lVert f_2'\rVert_2)\circ\widetilde{\phi}>>X_{A'}(2)\\
   @Vf_{\infty}/\lVert f_{\infty}\rVert_{\infty} V\mathrm{homeo.}V      
   @A {f_2'/\lVert f_2'\rVert_2} A\mathrm{homeo.} A \\
   (D^2,S^1)^{K_A}@>\widetilde{\phi}>\mathrm{homeo.}>(D^2,S^1)^{K_{A'}},
 \end{CD}
 \]
 where $f'_2\co\mathcal{S}_{A'}\to \mathcal{S}_{A'}$ is the function 
 of $L^2$-norm minima of Siegel leaves.
 By the commutativity,  
 it follows that
 $(f_2'/\lVert f_2'\rVert_2)\circ\widetilde{\phi}$ is
 a homeomorphism (see Theorem \ref{funda.thm.} and Proposition \ref{conti.sec.}), 
 whose inverse can be constructed by interchanging
 the roles of $A$ and $A'$, which is also smooth.
\end{proof}

In what follows we shall discuss everything 
with $\mathbb{C}^m$ 
replaced by its subspace $\mathbb{R}^m$.
In the foliation $\mathcal{F}$ given by the action \eqref{Def.F.}, 
a leaf $L_z$ lies in $\mathcal{S}_A\cap\mathbb{R}^m$ 
if and only if $z\in\mathcal{S}_A\cap\mathbb{R}^{m}$. 
Therefore all properties hold true when restricted to 
the real case.

We will still use the same notations as in the previous sections, 
with the exception that the notation 
$(D^1,S^0)^{K_A}$ is used for the associated \emph{real moment-angle complex}, 
i.e., the intersection of $(D^2,S^1)^{K_A}$ with $\mathbb{R}^m$ 
(see \eqref{M.A.C.} for details). 

Notice that the real version of Proposition \ref{rigi.} holds, 
namely the $\mathbb{Z}_2^m$-equivariant (where $\mathbb{Z}_2^m$ acts on 
$X_A(2)$ by changing the signs of coordinates) 
smooth structures on $X_A(2)$
are determined by combinatorial types of $K_A$. 
This can be deduced from a result of Wiemeler 
in \cite[Corollary 5.2]{Wie13} (see also \cite[Corollary 1.3]{Dav13}).

Recall that a subspace $X$ of $\mathbb{R}^m$ is a \emph{polyhedron}, 
if for every point $x\in X$ there is a compact set $C_x$, 
such that $x*C_x=\{ax+bl\mid l\in C_x,a+b=1,a,b\geq 0\}$
is a neighborhood of $x$ in $X$.
For instance, 
$(D^1,S^0)^{K_A}$ and $X_A(1)$ are polyhedra embedded in $\mathbb{R}^m$,
hence they can be triangulated (cf.~e.g.~\cite[Theorem 2.11]{RS72}). 

A polyhedron $X$ is a piecewise linear (abbr.~$\mathrm{PL}$) 
$n$-manifold if given certain triangulation, 
the link of each vertex is $\mathrm{PL}$ homeomorphic to the 
boundary of an $n$-simplex or to an $(n-1)$-simplex (i.e.  
these homeomorphims become simplicial after suitable subdivisions
on both sides). 
Note that this property is independent of the
triangulation chosen for $X$ (cf.~e.g.~\cite[pp.~20--22]{RS72}).

\begin{defin}[Whitehead triangulation]\label{smooth.PL.}
Let $X$ be a polyhedron and $M$ a smooth manifold. A map 
$\eta\co X\to M$ is a \emph{piecewise differentiable} 
(abbr.~$\mathrm{PD}$) homeomorphism 
if there exists a triangulation of $X$,
such that the restriction of $\eta$ to each simplex is smooth with 
the Jacobian matrix non-degenerate. 
Such a $\mathrm{PD}$ homeomorphism $\eta$ is called a 
\emph{Whitehead triangulation} of $M$, 
and also a \emph{smoothing} of $X$. 
\end{defin}

Note that by Propositions \ref{conti.} and \ref{conti.sec.}, 
the smooth function 
$f_2/\lVert f_2\rVert_2\co\mathcal{S}_A\to X_A(2)$ induces
a homeomorphism when restricted to either $(D^1,S^0)^{K_A}$
or $X_A(1)$. Moreover, the following lemma holds:
\begin{lem}\label{inj.lem.}
Let $A=(A_i)_{i=1}^m$ be an admissible tuple centered at 
the origin. 
If a space $Y\subset\mathbb{R}^m$ is either
\begin{enumerate}
 \item [(a)] the intersection of the $L^p$-link $X_A(p)$ (defined
	  by \eqref{Def.X.})
	  with the first orthant of $\mathbb{R}^m$ (i.e. points 
	  with non-negative coordinates), for any  
	  $p\geq 1$, or
 \item [(b)] a component of the polyhedral product 
	  $D(\sigma)=(D^1,S^0)^{\sigma}$ 
	  (see \eqref{M.A.C.} for definition, with the pair replaced), for
	  any $\sigma\in K_A$ with maximal dimension,
\end{enumerate}
then $Y$ is a smooth manifold with corners, and 
the differential of $f_2/\lVert f_2\rVert_2$ 
at any point of $Y$ induces a linear injection 
between corresponding tangent spaces.
\end{lem}
\begin{proof}
First we show that each $Y$ is indeed a smooth manifold with corners,
in both cases.
For (b) this is obvious since $Y$ is a cube of dimension 
$m-d-1$. As for (a), observe 
that for each $\sigma\in K_A$ with $\mathrm{card}(\sigma)=k$, 
the augmented subtuple $\widetilde{A}([m]\setminus\sigma)$ has 
rank $d+1$, where $\widetilde{A}=(\widetilde{A}_i)_{i=1}^m$ with 
$\widetilde{A}_i=(A_i^{\mathrm{T}},1)^{\mathrm{T}}$
(see Lemma \ref{adm.lem.}), therefore the row vectors of 
$\widetilde{A}$, together with canonical basis vectors $e_i\in\mathbb{R}^m$ 
(the vector with only $i$-th coordinate non-zero, which is one) 
for all $i\in\sigma$, form a matrix of rank $d+k+1$. 
This means that the intersection
\[ Y\bigcap_{i\in\sigma}F_i \]
is transverse, where 
$F_i=\{(x_i)_{i=1}^m\in\mathbb{R}^m\mid x_i=0\}$. 

Recall that
$\Phi_{A}(2)\co X_A(2)\times\mathbb{R}^d\times\mathbb{R}_{>0}
\to\mathcal{S}_A$ is a diffeomorphism, such that 
$f_2/\lVert f_2\rVert_2\circ\Phi_{A}(2)$ is the identity on $X_A(2)$ 
(see Theorem \ref{funda.thm.}).
Let   
\[
  \mathrm{d}\Phi_{A}(2)_x\co
   \mathbb{R}^{m-d-1}\times\mathbb{R}^d\times\mathbb{R}\to\mathbb{R}^{m}\in 
   T_{\Phi_{A}(2)(x)}\mathcal{S}_A,
\]
be the differential of $\Phi_{A}(2)$ at the point $x$, and
let $\zeta$ be the linear subspace $\{\bm{0}\}\times\mathbb{R}^d
\times\mathbb{R}\subset\mathbb{R}^m$ of dimension $d+1$. 
It suffices to show that for all $y=(y_i)_{i=1}^m\in Y$ with 
$x=(x_i)_{i=1}^m=f_2(y)/\lVert f_2(y)\rVert_2$,  
the intersection of the image of $\mathrm{d}\Phi_{A}(2)_{x}|_{\zeta}$ with 
the tangent space $T_yY$ is trivial.

For (a), note that from its definition \eqref{Def.X.}, 
the tangent space $T_yY$ is the orthogonal 
complement of the ($d+1$)-space 
spanned by the row vectors of the ($(d+1)\times m$)-matrix 
\[
  \widetilde{A}_{y^{p-1}}=((A^{\mathrm{T}}_i,1)^{\mathrm{T}}
y_i^{p-1})_{i=1}^m         
\]
and the image of $\mathrm{d}\Phi_{A}(2)_{x}|_{\zeta}$ is 
spanned by the row vectors of 
$\widetilde{A}_{y}=((A^{\mathrm{T}}_i,1)^{\mathrm{T}}
y_i)_{i=1}^m $. From the previous argument,  
the subtuple $\widetilde{A}_{y^{p-1}}(I_y)$ 
has rank $d+1$ ($I_y\subset [m]$ is the set of 
non-zero entries of $y$), hence any row vector of 
$\widetilde{A}_{y}(I_y)$ cannot be orthogonal to 
the corresponding one in $\widetilde{A}_{y^{p-1}}(I_y)$, 
otherwise itself must be zero 
(since we can write each $y_i^p$ as a square).

As for (b), the tangent space at $y\in (D^1,S^0)^{\sigma}$
is spanned by $\{e_i\mid i\in\sigma\}$, where 
$\mathrm{card}(\sigma)=m-d-1$. But we have shown
that the 
row vectors of $\widetilde{A}_{y}(I_y)$ and the 
basis of $T_yY$
has a full rank $m$, therefore the 
intersection of the image of $\mathrm{d}\Phi_{A}(2)_{x}|_{\zeta}$ with 
$T_yY$ must be trivial.
\end{proof}
As a corollary, we find that with given triangulations, 
the restriction of $f_2/\lVert f_2\rVert_2$ to 
either $(D^1,S^{0})^{K_A}$ or
 $X_A(1)$ will be a Whitehead triangulation
of $X_A(2)$.
By a theorem of Whitehead (cf.~\cite{Whi40}), 
if there is a $\mathrm{PD}$ homeomorphism from a polyhedron $X$ to 
a smooth manifold $M$, then $X$ is a $\mathrm{PL}$ manifold, 
and the $\mathrm{PL}$ structure on $X$ is uniquely determined by the 
smooth structure given on $M$.
Consequently, it follows that 
$(D^1,S^{0})^{K_A}$  and $X_A(1)$ are homeomorphic as 
$\mathrm{PL}$ manifolds.

At last, we make a conclusion to end this section. 
\begin{prop}  \label{fina.pl.}
  For each simplicial complex $K_A$ 
  induced from an admissible $m$-tuple $A$ centered at the origin, 
  there is a $\mathrm{PD}$
  homeomorphism from $(D^1,S^0)^{K_A}$ onto the smooth manifold 
  $X_A(2)$, thus $(D^1,S^0)^{K_A}$ is a 
  $\mathrm{PL}$ manifold of dimension $m-d-1$. If 
  $(D^1,S^0)^{K_A}$ has an exotic $PL$ structure, 
  then either it is not smoothable, or $X_A(2)$ must have
  different smooth structures. 
\end{prop}

%


\begin{thebibliography}{99}
\bibitem[BBCG10]{BBCG10}
A.~Bahri, M.~Bendersky, F.~R.~Cohen and S.~Gitler, \emph{The
  polyhedral product functor: a method of computation for
moment-angle complexes, arrangements and related spaces}, Adv.
Math.~\textbf{225} (2010), no.~3, 1634--1668.

\bibitem[Bas03]{Bas03}
I.~V.~Baskakov, \emph{Massey triple products in the cohomology of moment-angle complexes}, 
Uspekhi Mat.~Nauk \textbf{58} (2003), no.~5, 199–-200. 

\bibitem[BM06]{BM06} 
F.~Bosio and L.~Meersseman,
{\em Real quadrics in $\mathbb{C}^n$, complex manifolds and convex
polytopes}, Acta Mathematica \textbf{197} (2006), no.~1, 53--127. 
\MR{2285318}

\bibitem[BP02]{BP02} 
V.~Buchstaber and T.~Panov, {\em  Torus actions and their applications in topology and
combinatorics}, University Lecture Series, vol.~\textbf{24}, AMS (2002).

\bibitem[BP14]{BP14}
\bysame, \emph{Toric topology}, 
  arXiv preprint arXiv:1210.2368 (2014).


\bibitem[CKP78]{CKP78} 
C.~Camacho, N.~Kuiper and J.~Palls, 
{\em The topology of holomorphic flows with singularities}, 
Inst. Hautes \'{E}tudes Sci. Publ. Math. \textbf{48} (1978), 5--38.	

 \bibitem[DJ91]{DJ91}
M.~W.~Davis and T.~Januszkiewicz, {\em Convex polytopes, Coxeter orbifolds
and torus actions}, Duke Math.~J. \textbf{62} (1991), no.~2, 417--451. 
 
\bibitem[Dav13]{Dav13}
M.~W.~Davis, \emph{When are two Coxeter orbifolds diffeomorphic?},
arXiv preprint arXiv:1306.6946 (2013).

\bibitem[DS07]{DS07}
G.~Denham and A.~Suciu, 
\emph{Moment-angle complexes, monomial ideals and Massey products}, 
Pure Appl.~Math.~Q.~\textbf{3} (2007), no.~1, 25--60.

  \bibitem[GL13]{GL13}
S.~Gitler and S.~L\'{o}pez de Medrano,
{\em Intersections of quadrics, moment-angle manifolds and connected 
sums}, Geom.~Topol.~\textbf{17} (2013), no.~3, 1497--1534. 

 \bibitem[GM88]{GM88}
M.~Goresky and R.~MacPherson, 
\emph{Stratified morse theory}, Springer Berlin Heidelberg, 1988.

\bibitem[Gr\"{u}03]{Gru03}
B.~Gr\"{u}nbaum, \emph{Convex Polytopes}, second ed.,
Graduate Texts in Mathematics, vol.~\textbf{221} (2003),
Springer, New York. 

\bibitem[LdM89]{LdM89}
S. L\'{o}pez de Medrano, 
{\em Topology of the intersection of quadrics in $\mathbb{R}^n$}, 
in Algebraic Topology (Arcata CA, 1986), Springer Verlag 
Lecture Notes in Math.~\textbf{1370} (1989), 280--292.

\bibitem[LV97]{LV97}
S.~L\'{o}pez de Medrano and A.~Verjovsky,
\emph{A new family of complex, compact, non-symplectic manifolds},
Bol.~Soc.~Mat.~Braziliana \textbf{28} (1997), 243--267.

\bibitem[Mee00]{Mee00}
L.~Meersseman, \emph{A new geometric construction of
compact complex manifolds in any dimension}, Math.
Ann.~\textbf{317} (2000), 79--115.

\bibitem[MV04]{MV04}
  L.~Meersseman and A.~Verjovsky,  
  {\em Holomorphic principal bundles over projective toric varieties}, 
  J. Reine Angew.~Math.~\textbf{572} (2004), 57--96.
   
	 		
\bibitem[PU12]{PU12}
  T.~Panov and Y.~Ustinovsky, 
  {\em Complex-analytic structures on moment-angle manifolds},
  Moscow Math.~J.~\textbf{12} (2012), no.~1, 149--172.

\bibitem[Pan13]{Pan13}
  T.~Panov, \emph{Geometric structures on moment-angle manifolds}.
  Uspekhi Mat.~Nauk~\textbf{68} (2013), no.~3, 111--186 (Russian);
  Russian Math. Surveys~\textbf{68} (2013), no.~3, 503--568
  (English translation).
  
\bibitem[RS72]{RS72}
  C.~P.~Rourke and B.~J.~Sanderson, 
  \emph{Introduction to piecewise-linear topology}, 
  Ergebnisse der Mathematik und ihrer Grenzgebiete, 
  Band \textbf{69} (1972). Springer-Verlag, New York-Heidelberg.


\bibitem[Whi40]{Whi40}
  J.~H.~C.~Whitehead, 
  \emph{On $C^{1}$-Complexes}, Annals of Math.~\textbf{41} (1940), 
  809–-824.

\bibitem[Wie13]{Wie13}
  M.~Wiemeler, \emph{Exotic torus manifolds and equivariant smooth 
  structures on quasitoric manifolds}, Math.~Z.~\textbf{273} (2013), 
  1063--1084.
      
\end{thebibliography}
\end{document}